\documentclass[12pt]{amsart}
\usepackage{amsmath}
\usepackage{amsthm}
\usepackage{amssymb}
\usepackage{amsfonts}
\usepackage{enumerate}
\usepackage{extarrows}
\usepackage{geometry}
\usepackage{stmaryrd}

\usepackage[colorlinks,linkcolor=black,anchorcolor=black,citecolor=black]{hyperref}

\newtheorem{theorem}{Theorem}
\newtheorem{definition}[theorem]{Definition}
\newtheorem{corollary}[theorem]{Corollary}
\newtheorem{lemma}[theorem]{Lemma}
\newtheorem{landd}[theorem]{Lemma and Definition}
\newtheorem{remark}[theorem]{Remark}
\theoremstyle{definition}

\newcommand{\cc}{\mathbb{C}}
\newcommand{\ddb}{\partial\bar{\partial}}
\newcommand{\db}{{\bar{\partial}}}
\newcommand{\dd}{\partial}
\newcommand{\zb}{\bar{z}}
\newcommand{\cinf}{C^\infty}
\newcommand{\im}{\text{Im }}
\newcommand{\xx}{\mathcal{X}}
\newcommand{\xt}{X_t}
\newcommand{\h}[2][DR]{H^{#2}_{#1}}
\newcommand{\xc}{(X,\cc)}
\newcommand{\pp}[2]{\varphi^{#1\overline{#2}}}
\newcommand{\pq}{\varphi^1}
\newcommand{\pw}{\varphi^2}
\newcommand{\pe}{\varphi^3}
\newcommand{\pqb}{\bar{\varphi}^1}
\newcommand{\pwb}{\bar{\varphi}^2}
\newcommand{\peb}{\bar{\varphi}^3}

\title{Polarisation of SKT Calabi-Yau $\ddb$-manifolds by Aeppli classes}
\author{Yi Ma}

\begin{document}
	\maketitle
	\begin{abstract}
		Given a $\ddb$-manifold $X$ with trivial canonical bundle and carrying a metric $\omega$ such that $\ddb\omega=0$, we introduce the concept of small deformations of $X$ polarised by the Aeppli cohomology class $[\omega]_A$ of an SKT metric $\omega$. There is a correspondence between the manifolds polarised by $[\omega]_A$ in the Kuranishi family of $X$ and the Bott-Chern classes that are primitive in a sense that we define. We also investigate the existence of a primitive element in an arbitrary Bott-Chern primitive class and compare the metrics on the base space of the subfamily of manifolds polarised by $[\omega]_A$ within the Kuranishi family.
	\end{abstract}
	\section{Introduction}
	Let $X$ be a compact complex manifold of dimension $n$. Recall that a manifold is said to be a \textit{$\ddb$-manifold} if for any pure-type $ d $-closed form $ u $, we have the following equivalences:  \\
	$$u\text{ is }d\text{-exact}\Leftrightarrow u\text{ is }\partial\text{-exact} \Leftrightarrow u\text{ is }\db\text{-exact}\Leftrightarrow u\text{ is }\ddb\text{-exact.}$$
	\\A complex manifold is called \textit{Calabi-Yau} if its canonical bundle $K_X$ is trivial.\smallskip
	
	A Hermitian metric $\omega$, seen as a positive definite $\cinf \ (1,1)$-form, on a complex manifold $X$ is called strong K\"ahler with torsion (SKT for short) if $\ddb\omega=0$ and it is called Hermitian-symplectic \cite{streets2010parabolic} if $\omega$ is the component of bidegree $(1,1)$ of a real smooth $d$-closed $2$-form on $X$. Obviously, on a $\ddb$-manifold, a metric is SKT if and only if it is Hermitian-symplectic. The study of SKT metrics (also called pluriclosed, see \cite{streets2010parabolic}) has received a lot of attention over recent years. A necessary condition for the existence of a smooth family of SKT metrics on a differentiable family of complex manifolds is given in \cite{DefofSKT}. The existence of a left-invariant SKT structure on any even-dimensional compact Lie group G is obtained in \cite{MR2797819}.\smallskip
	
	The $\ddb$-property is open under holomorphic deformations of the complex structure by \cite{wu2006geometry}. Namely, if $\pi:\xx\rightarrow B$ is a proper holomorphic submersion between complex manifolds, and $X_0:=\pi^{-1}(0)$ is a $\ddb$-manifold, then $X_t:=\pi^{-1}(t)$ is a $\ddb$-manifold for all $t$ in a small neighbourhood of $0$. Moreover, the Hodge numbers are independent of $t$ in this case. If $X_0$ is a Calabi-Yau manifold and $h^{n,0}$ does not jump, which means that $h^{n,0}(t)=h^{n,0}(0)$ for $t$ in a small neighbourhood of $0$, then $X_t$ is again a Calabi-Yau manifold for $t$ close to $0$. Though the SKT condition is not deformation open, if $X_0$ is an SKT $\ddb$-manifold, then $X_t$ is again an SKT $\ddb$-manifold for $t$ in a small neighbourhood of $0$. Indeed, the Hermitian-symplectic condition is deformation open by \cite{SongYang} (see also \cite{Bellitir}). Putting these things together, we get that if $X_0$ is an SKT Calabi-Yau $\ddb$-manifold, then for all $t$ close to $0$, $X_t$ is again an SKT Calabi-Yau $\ddb$-manifold.\smallskip

	Now, fix an SKT metric $\omega$ on a Calabi-Yau SKT $\ddb$-manifold $X$. Let $(X_t)_{t\in B}$ be the Kuranishi family of $X$. By \cite{bogomolov1978hamiltonian}, \cite{tian1987smoothness}, \cite{Tod89} (see also \cite{popovici2019holomorphic}), the Kuranishi family is unobstructed. In particular, $B$ can be seen as a ball about $0$ in $\h[]{0,1}(X_0,T^{1,0}X_0)$, where $T^{1,0}X_0$ is the holomorphic tangent bundle of $X_0$. We can define the notion of $X_t$ being polarised by the SKT Aeppli class $[\omega]_A$ by requiring the canonical image $\{\omega\}_{DR}$ of $[\omega]_A$ in $\h{2}\xc$, where $X$ is the $C^\infty$ manifold underlying the fibres $X_t$, to be of type $(1,1)$ for the complex structure of $\xt$ (see Definition \ref{def1}). Somehow we can view the set $B_{[\omega]}$ of all the fibres $X_t$ polarised by $[\omega]_A$ as the intersection of $B$ and $H^{0,1}(X,T^{1,0}X)_{[\omega]}$, the latter being defined in Lemma \ref{lem3} as a vector subspace of $\h[]{0,1}(X_0,T^{1,0}X_0)$. 
	
	In section \ref{secprim}, we define the primitivity of certain Bott-Chern classes w.r.t. $[\omega]_A$. Then we can identify the space of $H^{0,1}(X_0,T^{1,0}X_0)_{[\omega]}$ with the space of Bott-Chern primitive classes. 
	
	In section \ref{secmetrics}, we compare the Weil-Petersson metric and the metric induced by the period map on the base space $B_{[\omega]}$ of the family of $[\omega]$-polarised small deformations of $X$. In the case of K\"ahler polarised deformations, these two metrics coincide with each other  by \cite[Theorem 2]{tian1987smoothness}.
	
	This work was inspired by \cite{popovici2019holomorphic} where the notion of small deformations co-polarised by a balanced class was introduced and studied. Besides, deformations co-polarised by a Gauduchon class was studied in \cite{unknown}.
	
	\smallskip
	\textbf{Acknowledgements.} The author is supported by China Scholarship Council. He	is grateful to his Ph.D. supervisor Dan Popovici for suggesting the problem and for his constant guidance and support.
	\section{Preliminaries}
		Recall that a compact complex manifold $X$ is a $\ddb$-manifold if and only if there are canonical isomorphisms between the Bott-Chern, Aeppli and Dolbeaut cohomology groups, i.e. the canonical maps
		{\footnotesize\begin{align*}
			&H_{BC}^{p,q}(X,\cc)=\dfrac{\ker\partial\cap\ker\db}{\text{Im}\ddb}&\longrightarrow& H_{\db}^{p,q}(X,\cc)=\dfrac{\ker\db}{\text{Im}\db}&\longrightarrow& H_{A}^{p,q}(X,\cc)=\dfrac{\ker\ddb}{\text{Im}\partial+\text{Im}\db}\\
			&\hspace{5em}[\alpha]_{BC}&\longmapsto&\hspace{5em}[\alpha]_{\partial}&\longmapsto&\hspace{5em}[\alpha]_{A}
		\end{align*}}are isomorphisms for all $0\le p,q\le n$. For $\ddb$-manifolds, we have a Hodge decomposition in the sense that the canonical map with \\$d\alpha^{p,q}=0, \forall p,q=k-p$
		\begin{align*}
			\bigoplus_{p+q=k}\h[A]{p,q}\xc\cong\h[DR]{k}\xc\\
			([\alpha^{p,q}]_A)_{p+q=k}\mapsto\{\sum_{p+q=k}\alpha^{p,q}\} \qquad
		\end{align*}
		is an isomorphism for all $0\le k\le 2n$, where all the forms $\alpha^{p,q}$ are $d$-closed. The $\ddb$-assumption on $X$ guarantees that every Aeppli cohomology class can be represented by a $d$-closed form. We could use any of the Bott-Chern, Aeppli or Dolbeaut cohomologies here. 
		As in \cite{10.2307/1969879} and \cite{schweitzer2007autour}, Bott-Chern and Aeppli Laplacians $\Delta_{BC}, \Delta_{A}:\cinf_{p,q}\xc\rightarrow\cinf_{p,q}\xc$ are defined as:
		\begin{align*}
			\Delta_{BC}=\dd^*\dd+\db^*\db+(\ddb)(\ddb)^*+(\ddb)^*(\ddb)+(\dd^*\db)(\dd^*\db)^*+(\dd^*\db)^*(\dd^*\db),\\
		\Delta_A=\dd\dd^*+\db\db^*+(\ddb)(\ddb)^*+(\ddb)^*(\ddb)+(\dd\db^*)(\dd\db^*)^*+(\dd\db^*)^*(\dd\db^*),
		\end{align*}
		and then we have
		\begin{align}
		\label{kerBC}	\ker\Delta_{BC}=\ker\dd\cap\ker\db\cap\ker(\ddb)^*,\\
		\label{kerA}	\ker\Delta_{A}=\ker(\ddb)\cap\ker\dd^*\cap\ker\db^*.
		\end{align}
		Let $\omega$ be an SKT metric on a $\ddb$-manifold $X$. By the $\ddb$-property, there exists a form $\alpha\in\cinf_{0,1}(X,\cc)$, such that
		\begin{equation}
			\label{defalpha}\db\omega=\ddb\alpha.
		\end{equation}
		
		Note that $d(\omega+\dd\alpha+\db\bar{\alpha})=0$, so $\omega+\dd\alpha+\db\bar{\alpha}$ is a $d$-closed representative of  $[\omega]_A$. 
		We define $\{\omega\}_{DR}$ (resp. $[\omega]_\db$) to be the image of $[\omega]_A$ under the canonical injection $\h[A]{1,1}\xc\hookrightarrow\h{2}\xc$ (resp. the isomorphism $\h[A]{1,1}\xc\overset{\cong}\rightarrow\h[\db]{1,1}\xc$), which means that $\{\omega\}_{DR}=\{\omega+\dd\alpha+\db\bar{\alpha}\}_{DR}$ (resp. $[\omega]_\db=[\omega+\dd\alpha]_\db$).

			Given a Calabi-Yau manifold $X$, we fix a non-vanishing holomorphic $n$-form $u$ on $X$. Note that $u$ exists and is unique up to a multiplicative constant since $K_X$ is trivial. It defines the \textit{Calabi-Yau isomorphism}:
		\begin{align*}
			T_{[u]}:H^{0,1}(X,T^{1,0}X)&\longrightarrow H_\db^{n-1,1}(X,\cc)\\
			[\theta]&\longmapsto[\theta\lrcorner u],
		\end{align*}
		where the operator $\cdot\lrcorner\cdot$ combines the contraction of $u$ by the vector field component of $\theta$ with the multiplication by the $(0,1)$-form component.

		For a primitive form $v$ of bidegree $(p,q)$, we will often use the following formula (see \cite[Proposition 6.29]{voisin_2002}) 
		\begin{equation}
			\star v=(-1)^{\frac{(p+q)(p+q+1)}{2}}i^{p-q}\dfrac{\omega^{n-p-q}\wedge v}{(n-p-q)!} \label{voiprim}
		\end{equation}
		for the Hodge star operator $\star$.
		
	\section{Polarisation by SKT classes}
		Let $X$ be a compact SKT Calabi-Yau $\ddb$-manifold and let $\omega$ be an SKT metric on it. 
		Let $\pi:\xx\rightarrow B$ be the Kuranishi family of $X$. In a small neighbourhood of $0$, $X_t:=\pi^{-1}(t)$ is again a Calabi-Yau SKT $\ddb$-manifold, and we have the following Hodge decomposition by \cite{schweitzer2007autour}
		\begin{equation}
				\h{2}(X,\cc)\simeq\h[A]{2,0}(\xt,\cc)\oplus\h[A]{1,1}(\xt,\cc)\oplus\h[A]{0,2}(\xt,\cc),\quad t\sim0,\label{hd}
		\end{equation}
		where "$\cong$" stands for the canonical isomorphism whose inverse is defined by $([\alpha^{2,0}]_A,[\alpha^{1,1}]_A,[\alpha^{0,2}]_A)\mapsto\{\alpha^{2,0}+\alpha^{1,1}+\alpha^{0,2}\}_{DR}$, where $d\alpha^{p,2-p}=0$ for $p=0,1,2$.		
		\begin{definition}\label{def1}
			Fix the Aeppli class $[\omega]_A\in H^{1,1}_A(X,\cc)$ of an SKT metric $\omega$ on $X_0=X$. For $t\in B$, we say that $X_t$ is polarised by $[\omega]_A$ if the projection $[\omega]_{A,t}^{0,2}$ of $\{\omega\}_{DR}$ onto $H^{0,2}_A(X_t,\cc)$ w.r.t (\ref{hd}) is $0$.
			
			Denote by $B_{[\omega]}$ the set of $t\in B$ such that $X_t$ is polarised by $[\omega]_A$, namely
			\[B_{[\omega]}=\{t\in B|[\omega]_{A,t}^{0,2}=0\in\h[A]{0,2}(\xt,\cc)\}.
			\]
		\end{definition}
			This means that  $\{\omega\}_{DR}$ is of $J_t$-pure-type $(1,1)$ for $t\in B_{[\omega]}$ since $[\omega]_{A,t}^{2,0}=0$ if and only if $[\omega]_{A,t}^{0,2}=0$. Indeed, $\{\omega\}_{DR}$ being real, $[\omega]_{A,t}^{2,0}$ is the conjugate to $[\omega]_{A,t}^{0,2}$.
		
		\begin{theorem}
			
			Let $\omega$ be an SKT metric on a compact $\ddb$-manifold $X$ and let $\pi:\xx\rightarrow B$ be its Kuranishi family. Consider $\gamma^{1,1}_t\in\h[A]{1,1}(\xt,\cc)$ the Aeppli component of $J_t$-type $(1,1)$ of $\{\omega\}_{DR}$ w.r.t. (\ref{hd}). Then there exists an SKT metric $\omega_t\in\gamma^{1,1}_t$ for all $t$ in a small neighbourhood of $0$.
		\end{theorem}
		\begin{proof}

			By the definition (\ref{defalpha}) of $\alpha$, we construct two $d$-closed forms:
			\begin{align*}
				&\tilde{\omega}=\omega+\dd\alpha+\db\bar{\alpha},\\
				&\hat{\omega}=-\dd\bar{\alpha}+\omega-\db\alpha.
			\end{align*}
		We know that $\tilde{\omega}$ is of $J_0$-bidegree $(1,1)$ and $\omega$ is the $J_0$-$(1,1)$-component of the 2-form $\hat{\omega}$. Since $\tilde{\omega}=\hat{\omega}+d(\alpha+\bar{\alpha})$, we have that $\{\tilde{\omega}\}_{DR}=\{\hat{\omega}\}_{DR}$
		
		Decompose $\hat{\omega}$ into components of pure $J_t$-type:
		\[\hat{\omega}=\Omega_t^{2,0}+\omega_t^{1,1}+\Omega_t^{0,2}.\]
		We know that $\omega_0^{1,1}=\omega$. Moreover, $\omega_t^{1,1}$ is real and $\dd_t\db_t$-closed because $\hat{\omega}$ is real and $d$-closed. By the continuity of $(\omega^{1,1}_t)_t$ with respect to $t$ and $\omega_0^{1,1}=\omega>0$, we also have $\omega_t^{1,1}>0$ for $t$ near $0$. This ensures that $\omega^{1,1}_t$ is an SKT metric on $\xt$. Moreover, we have $\omega_t^{1,1}\in\gamma_t^{1,1},\ \forall t\sim0$. We put $\omega_t:=\omega_t^{1,1},\ \forall t\sim0$ and we are done.
	\end{proof}

		Recall that $\h[]{0,1}(X_0,T^{1,0}X_0):=${\footnotesize$\dfrac{\ker\db:\cinf_{0,1}(X_0,T^{1,0}X_0)\rightarrow\cinf_{0,2}(X_0,T^{1,0}X_0)}{\text{Im}\db:\cinf_{0,0}(X_0,T^{1,0}X_0)\rightarrow\cinf_{0,1}(X_0,T^{1,0}X_0)}$}, where $\db$ is the holomorphic structure of $T^{1,0}X_0$. Because the Kodaira-Spencer map gives an isomorphism between $T_0B$ and $\h[]{0,1}(X_0,T^{1,0}X_0)$, after possibly
		shrinking $B$, we can view the base space $B$ as an open subset of $\h[]{0,1}(X_0,T^{1,0}X_0)$. Then we have the following lemma:
		
		\begin{lemma}\label{lem3}
			Consider the following subspace of $H^{0,1}(X,T^{1,0}X)$:
			\begin{align*}
				H^{0,1}(X,T^{1,0}X)_{[\omega]}:&=\{[\theta]\in H^{0,1}(X,T^{1,0}X)|[\theta\lrcorner\zeta]_A=0\in H_A^{0,2}\xc\}\\
				&=\{[\theta]\in H^{0,1}(X,T^{1,0}X)\arrowvert[\theta\lrcorner\zeta]_\db=0\in H_{\db}^{0,2}\xc\},
			\end{align*}
			where $\zeta$ is an arbitrary representative in $[\omega]_\db$. It is well defined and
			\[T^{1,0}_0B_{[\omega]}= H^{0,1}(X,T^{1,0}X)_{[\omega]}.\]
			Or locally we can view this as
			$$B_{[\omega]}=B\cap H^{0,1}(X,T^{1,0}X)_{[\omega]}.$$
		\end{lemma}
	\begin{proof}
		By \[\db(\theta\lrcorner\beta)=\db\theta\lrcorner\beta+(-1)^q\theta\lrcorner\db\beta,\quad \forall\theta\in\cinf_{0,q}(X,T^{1,0}X), \beta\in\cinf_{1,q'}\xc,
		\]
		we see that for $\theta$ a representative in $[\theta]$, $\db(\theta\lrcorner\zeta)=0$ 
		and
		\begin{align*}
			\theta\lrcorner(\zeta+\db\zeta')=\theta\lrcorner\zeta+\db(\theta\lrcorner\zeta')\\
			(\theta+\db\theta')\lrcorner\zeta=\theta\lrcorner\zeta+\db(\theta'\lrcorner\zeta)
		\end{align*}
		for $\theta'\in\cinf(X,T^{1,0}X),\zeta'\in\cinf_{1,0}\xc$. Hence the classes $[\theta\lrcorner\zeta]_\db$ and $[\theta\lrcorner\zeta]_A$ are independent of the choices of representatives of $[\theta]$ and $[\zeta]_\db$. Therefore, the space $H^{0,1}(X,T^{1,0}X)_{[\omega]}$ is well defined.
		 By the $\ddb$-property, the classes $[\theta\lrcorner\zeta]_A$ and $[\theta\lrcorner\zeta]_\db$ correspond to each other under the isomorphism $\h[A]{0,2}\xc\overset{\cong}{\longrightarrow}\h[\db]{0,2}\xc$, so we also have
		\begin{align*}
			&\{[\theta]\in H^{0,1}(X,T^{1,0}X)|[\theta\lrcorner\zeta]_A=0\in H_A^{0,2}\xc\}\\
			=&\{[\theta]\in H^{0,1}(X,T^{1,0}X)\arrowvert[\theta\lrcorner\zeta]_\db=0\in H_{\db}^{0,2}\xc\}
		\end{align*}
		For $t$ near $0$, $X_t$ is a $\ddb$-manifold, so we have the Hodge decompositions:
		\begin{align*}
			\h{2}(X,\cc)=&\h[\db]{2,0}(\xt,\cc)\oplus\h[\db]{1,1}(\xt,\cc)\oplus\h[\db]{0,2}(\xt,\cc)\\
			\simeq&\h[A]{2,0}(\xt,\cc)\oplus\h[A]{1,1}(\xt,\cc)\oplus\h[A]{0,2}(\xt,\cc).
		\end{align*}
		Take a vector in $T^{1,0}_0B_{[\omega]}$, say $\frac{\dd}{\dd t_i}|_{t=0}$. Denote by $[\theta]$ the image of it under the Kodaira-Spencer map $\rho:T^{1,0}_0B\rightarrow\h[]{0,1}(X_0,T^{1,0}X_0)$. Then we have $\nabla_{\frac{\dd}{\dd t_i}|_{t=0}}[\omega]_{A,t}^{0,2}=[\theta\lrcorner\zeta]_A$, where $\nabla$ is the Gauss-Manin connection. Then by the definition of $B_{[\omega]}$, we have
		\begin{align*}
			T^{1,0}_0B_{[\omega]}=&\{[\theta]\in H^{0,1}(X,T^{1,0}X)|[\theta\lrcorner\zeta]_A=0\in H_A^{0,2}\xc\}\\=& H^{0,1}(X,T^{1,0}X)_{[\omega]}.
		\end{align*}
	\end{proof}
		\begin{remark}
			If $\omega$ is moreover K\"ahler, $\alpha$ can be taken as $0$. Therefore everything here coincides with the case of K\"ahler polarised deformation.
		\end{remark}

	\section{Primitive classes}\label{secprim}
			\begin{landd}\label{def5}
				Let $X$ be a compact complex manifold of dimension $n$, and $\omega$ be an SKT metric on $X$. Then the map
				\begin{align*}
					L_{[\omega]}:H^{p,q}_{BC}(X,\cc)&\longrightarrow H^{p+1,q+1}_A(X,\cc)\\
					[\gamma]_{BC}&\longmapsto [\omega\wedge\gamma]_A
				\end{align*}
				is well-defined and only depends on the Aeppli class of $\omega$.
				
				
				We say that a Bott-Chern class $[\gamma]_{BC}$ of bidegree $(p,n-p)$ is \textit{primitive} (or $[\omega]_A$-primitive) if $L_{[\omega]}([\gamma]_{BC})=0$. We denote the space of primitive Bott-Chern classes of bidegree $(n-1,1)$ by
				$H^{n-1,1}_{BC,prim}(X,\cc)$.
				
			\end{landd}
			\begin{proof}
				Since $\partial\gamma=\db\gamma=0$, we get \[\ddb(\omega\wedge\gamma)=\ddb\omega\wedge\gamma=0,\] so $\omega\wedge\gamma$ represents an Aeppli class.
				By
				\begin{align*}
					\omega\wedge(\gamma+\ddb\beta)=&\omega\wedge\gamma+\partial(\omega\wedge\db\beta)+\db(\partial\omega\wedge\beta)+\ddb\omega\wedge\beta\\=&\omega\wedge\gamma+\partial(\omega\wedge\db\beta)+\db(\partial\omega\wedge\beta),
				\end{align*}
				we have $[\omega\wedge(\gamma+\ddb\beta)]_A=[\omega\wedge\gamma]_A$. Hence the map $L_{[\omega]}$ is well-defined.
				From
				 \[(\omega+\partial\beta_1+\db\beta_2)\wedge\gamma=\omega\wedge\gamma+\partial(\beta_1\wedge\gamma)+\db(\beta_2\wedge\gamma),\]
				we see that the map $L_{[\omega]}$ only depends on the Aeppli class of $\omega$.
			\end{proof}

		Let $X$ be a $\ddb$-manifold. We denote by \[j:H_A^{p,q}(X,\cc)\overset{\cong}{\longrightarrow}H^{p,q}_\db(X,\cc)\]
		the canonical isomorphism and by
		\begin{equation*}
			\widetilde{T_{[u]}}:H^{0,1}(X,T^{1,0}X)\xlongrightarrow[\cong]{T_{[u]}} H^{n-1,1}_\db(X,\cc)\xlongrightarrow[\cong]{i} H^{n-1,1}_{BC}(X,\cc)
		\end{equation*}
		the composition of canonical isomorphism $i$ and Calabi-Yau isomorphism.
		
		More precisely, for $[\theta]\in H^{0,1}(X,T^{1,0}X)$, we have $T_{[u]}([\theta])=[\theta\lrcorner u]\in\h[\db]{n-1,1}\xc$. Then the isomorphism $i$ maps $[\theta\lrcorner u]$ to $[\theta\lrcorner u+\db \eta]_{BC}$, where $\eta$ is a $(n-1,0)$-form such that $\partial(\theta\lrcorner u+\db \eta)=0$. Such a form $\eta$ exists because of the $d$-closedness, $\dd$-exactness of $\dd(\theta\lrcorner u)$ and the $\ddb$-property of $X$. The class $[\theta\lrcorner u+\db \eta]_{BC}$ is independent of the choice of $\eta$ such that $\partial(\theta\lrcorner u+\db \eta)=0$, again by the $\ddb$-property of $X$.
		
		\begin{lemma}\label{lemma8}
			The following map
			\begin{align*}
				f_{[u]}:H^{0,q}_\db(X,\cc)&\longrightarrow H^{n,q}_\db(X,\cc)\\
				[\xi]&\longmapsto[ u\wedge\xi]
			\end{align*}
			is well-defined and an isomorphism for all $q=1,\cdots,n$. 
			As a consequence, we have the equality between the Hodge numbers $h^{0,q}=h^{n,q}$.
		\end{lemma}
		\begin{proof}
			To check that this is an isomorphism, we first check that
			\begin{align*}
				f_u:C_{0,q}^\infty(X,\cc)&\longrightarrow C_{n,q}^\infty(X,\cc)\\
				\xi&\longmapsto u\wedge\xi
			\end{align*}
			is an isomorphism. In local coordinates, let 
			\begin{equation*}
				\xi=\sum_{\lvert J\rvert=q}\xi_{\bar{J}}d\zb_J\ \text{ and}\ \ u=gdz_1\wedge\cdots\wedge dz_n
			\end{equation*}
			where $g$ does not vanish. 
			\[
			u\wedge\xi=\sum_{\lvert J\rvert=q}g\xi_{\bar{J}}dz_1\wedge\cdots\wedge dz_n\wedge d\zb_J.
			\]
			Hence 
			\begin{align*}
				C_{0,q}^\infty(X,\cc)&\longrightarrow C_{n,q}^\infty(X,\cc)\\
				\xi&\longmapsto u\wedge\xi
			\end{align*}
			is an isomorphism because $g$ does not vanish.
			
			It is easy to check $f_u(\ker\db)= \ker\db$ and $f_u(\im\db)\subset \im\db$ since $\db u=0$, which means that $f_{[u]}$ is well-defined and injective. Now we already have $h^{n,q}\ge h^{0,q}$ for all $q=1,\cdots, n$ by injectivity. Take any $\db$-exact form
			$\db\eta\in C_{n,q}^\infty(X,\cc)$. In local coordinates, write
			\[
			\eta=\sum_{\lvert I\rvert=q-1}\eta_{\bar{I}}dz_1\wedge\cdots\wedge dz_n\wedge d\zb_I
			\]
			Let $\zeta=\sum_{\lvert I\rvert=q-1}\frac{\eta_{\bar{I}}}{g}d\zb_I$, it is easy to check that $u\wedge\db\zeta=\db \eta$ and this implies the surjectivity of $f_{[u]}$.
			
		\end{proof}
		\begin{theorem}\label{thm}
			If $\omega$ is an SKT metric on a Calabi-Yau $\ddb$-manifold $X$, then 
			\begin{equation*}
				\widetilde{T_{[u]}}:H^{0,1}(X,T^{1,0}X)_{[\omega]}\longrightarrow H^{n-1,1}_{BC,prim}(X,\cc)
			\end{equation*}
			is an isomorphism.
		\end{theorem}
		\begin{proof}
			For $[\theta]\in H^{0,1}(X,T^{1,0}X)$, we have  $\widetilde{T_{[u]}}([\theta])=[\theta\lrcorner u+\db \eta]_{BC}$, where $\eta$ is some $(n-1,0)$-form such that $\partial(\theta\lrcorner u+\db \eta)=0$. 
			
			Let $\zeta$ be an arbitrary representative in $[\omega]_\db$. Then there exist $\beta_1\in\cinf_{0,1}\xc$ and $\beta_2\in\cinf_{1,0}\xc$, such that $\zeta-\omega=\dd\beta_1+\db\beta_2$.
			By $$0=\theta\lrcorner(\omega\wedge u)=(\theta\lrcorner\omega)\wedge u+\omega\wedge(\theta\lrcorner u),$$ we get 
			\begin{align*}
				L_{[\omega]}(\widetilde{T_{[u]}}([\theta]))=&[\omega\wedge(\theta\lrcorner u+\db \eta)]_A\\
				=&[-(\theta\lrcorner\omega)\wedge u+\omega\wedge\db\eta]_A\\
				=&[(\theta\lrcorner(\zeta-\omega))\wedge u+\omega\wedge\db\eta]_A-[(\theta\lrcorner\zeta)\wedge u]_A.
			\end{align*}
			Moreover,
			\begin{align*}
				[(\theta\lrcorner(\zeta-\omega))\wedge u+\omega\wedge\db\eta]_A
				=&[-(\zeta-\omega)\wedge(\theta\lrcorner u)+\omega\wedge\db\eta]_A\\
				=&[-(\dd\beta_1+\db\beta_2)\wedge(\theta\lrcorner u)+\omega\wedge\db\eta]_A\\
				=&[\beta_1\wedge\dd(\theta\lrcorner u)+\omega\wedge\db\eta]_A\\
				=&[-\beta_1\wedge\ddb\eta+\omega\wedge\db\eta]_A\\
				=&[(\omega+\dd\beta_1+\db\beta_2)\wedge\db\eta]_A\\
				=&[\zeta\wedge\db\eta]_A\\
				=&0.
			\end{align*}
			Hence we have
			\begin{align*}
				(j\circ L_{[\omega]})(\widetilde{T_{[u]}}([\theta]))&=j([\omega\wedge(\theta\lrcorner u+\db \eta)]_A)\\&=[-(\theta\lrcorner\zeta)\wedge u]_\db.
			\end{align*} By Lemma \ref{lemma8}, we have that $(j\circ L_{[\omega]})(\widetilde{T_{[u]}}([\theta]))=0$ if and only if $[\theta]\in H^{0,1}(X,T^{1,0}X)_{[\omega]}$, i.e. $\widetilde{T_{[u]}}$ is an isomorphism.
		\end{proof}
		\begin{corollary}
			
				The map
			\begin{align*}
				L_{[\omega]}:H^{n-1,1}_{BC}(X,\cc)&\longrightarrow H^{n,2}_A(X,\cc)\\
				[\gamma]_{BC}&\longmapsto [\omega\wedge\gamma]_A
			\end{align*}
			is surjective.
		\end{corollary}
			\begin{proof}
				For $\ddb$-manifolds, we know the equality between the dimensions of corresponding Dolbeaut, Bott-Chern, Aeppli cohomology classes $h^{p,q}:=h^{p,q}_\db=h^{p,q}_{BC}=h^{p,q}_A$.
				Therefore we have 
			\begin{align*}
				\dim_\cc H^{0,1}(X,T^{1,0}X)_{[\omega]}&=\dim_\cc H^{0,1}(X,T^{1,0}X)-\dim_\cc \h[]{0,2}\xc\\&=h^{n-1,1}-h^{0,2}
			\end{align*}
			by the surjectivity of 
			\begin{align*}
				H^{0,1}(X,T^{1,0}X)&\rightarrow\h[\db]{0,2}\xc\\
				[\theta]&\mapsto[\theta\lrcorner\omega].
			\end{align*}
		
			To prove this surjectivity, we write the map in local coordinates $(z_1,\cdots,z_n)$. The metric can be written as $\omega=\sum_{j,k=1}^{n}i\omega_{j\bar{k}}dz_i\wedge d\zb_k$.
			Then the map is $$[\theta=\sum_{i,j=1}^{n}\theta_j^id\zb_i\otimes\frac{\dd}{\dd z_j}]\mapsto[\theta\lrcorner\omega=\sum_{i,j,k=1}^{n}i\theta_j^i\omega_{j\bar{k}}d\zb_i\wedge d\zb_k].$$
			The surjectivity is derived from the invertibility of $(\omega_{j\bar{k}})_{1\le j,k\le n}$.
			 
			By Theorem \ref{thm}, we get
			\[\dim_\cc\ker L_{[\omega]}=\dim_\cc\h[BC,prim]{n-1,1}\xc=\dim_\cc H^{0,1}(X,T^{1,0}X)_{[\omega]}=h^{n-1,1}-h^{0,2}.
			\]
			Hence we have$$\dim_\cc \text{Im}L_{[\omega]}=\dim_\cc\h[BC]{n-1,1}\xc-\dim_\cc\ker L_{[\omega]}=h^{0,2}=\dim_\cc H^{n,2}_A(X,\cc),$$ which means that $L_{[\omega]}|_{\h[BC]{n-1,1}\xc}$ is surjective.
			\end{proof}

		In the K\"ahler case, every primitive Dolbeaut class has one and only one $d$-closed primitive representative, which is the $\Delta''$-harmonic element. In general, we have the following:
		
		\begin{lemma}
			For a primitive form $v$ of degree $n$, the following are equivalent: 
		\begin{enumerate}[\rm(a)]
			\item $d$-closed,
			\item $d^*$-closed,
			\item $\Delta v=0$,
			\item $\Delta_Av=0$,
			\item $\Delta_{BC}v=0$.
		\end{enumerate}
		\end{lemma}
		\begin{proof}
			Recall that
			\begin{equation*}
				\dd^*=-\star\db\star,  \db^*=-\star\dd\star,  d^*=-\star d\star.
			\end{equation*}
			Therefore by formula (\ref{voiprim}), $v$ and $\star v$ are proportional. Hence, (a) and (b) are equivalent; $\dd v=0$ if and only if $\db^*v=0$; $\dd^* v=0$ if and only if $\db v=0$. Thus by equations (\ref{kerBC}) and (\ref{kerA}), we have that (c), (d) and (e) are equivalent. The equivalence between (a) and (c) is trivial now.
		\end{proof}
		
		In the SKT case, there is at most one $d$-closed primitive form in a cohomology class of degree $n$. If we look for a $d$-closed primitive form in a primitive class on an SKT $\ddb$-manifold, it does not matter whether we search in a Dolbeaut or Bott-Chern class.\smallskip

		We analyse an example of SKT manifold given in \cite{tardini2017geometric} from one point of view.
		Let $X=S^3\times S^3$. Recall that the 3-sphere is diffeomorphic to the special unitary group $SU(2)$. We know that $\mathfrak{su}(2)$,  the Lie algebra of $SU(2)$, has a basis $\{e_1,e_2,e_3\}$ with the following relations:
		\[[e_1.e_2]=2e_3,\quad[e_1,e_3]=-2e_2,\quad[e_2,e_3]=2e_1.\]
		Then by the Cartan formula, we have the following for the dual co-frame $\{e^1,e^2,e^3\}$:
		\[de^1=-2e^2\wedge e^3, \quad de^2=2e^1\wedge e^3 \quad de^3=-2e^1\wedge e^2.\]
		On $X$, we take $\{e_1,e_2,e_3\}$ and $\{f_1,f_2,f_3\}$ to be two copies of this basis of $\mathfrak{su}(2)$, and the corresponding co-frames $\{e^1,e^2,e^3\}$ and $\{f^1,f^2,f^3\}$. 
		Then we define a complex structure on $X$, namely the \textit{Calabi-Eckmann complex structure} by:
		\[Je_1=e_2,\quad Jf_1=f_2, \quad Je_3=f_3.\]
		We have
		\[Je^1=-e^2,\quad Jf^1=-f^2, \quad Je^3=-f^3.\]
		For the complex co-frame of $(1,0)$-forms, we set
		\[\varphi^1=e^1+ie^2,\quad \varphi^2=f^1+if^2, \quad \varphi^3=e^3+if^3.\]
		Thus, we have
		\begin{align*}
			d\varphi^1&=i\pq\wedge\pe+i\pq\wedge\bar{\varphi}^3,\\
			d\pw&=\pw\wedge\pe-\pw\wedge\peb,\\
			d\pe&=-i\pq\wedge\pqb+\pw\wedge\pwb.
		\end{align*}
		Equivalently,
		\begin{align*}
			\dd\varphi^1&=i\pq\wedge\pe,\\
			\dd\pw&=\pw\wedge\pe,\\
			\dd\pe&=0,\\
			\db\pq&=i\pq\wedge\bar{\varphi}^3,\\
			\db\pw&=-\pw\wedge\peb,\\
			\db\pe&=-i\pq\wedge\pqb+\pw\wedge\pwb.
		\end{align*}
		We define a Hermitian metric
		\[\omega:=\frac{i}{2}\sum_{j=1}^{3}\varphi^j\wedge\bar{\varphi}^j.\]
		By direct calculation we know that $\ddb\omega=0$, which means that $\omega$ is an SKT metric on $X$.
		We calculate the Bott-Chern cohomology groups (see \cite{tardini2017geometric}):
		\begin{align*}
			\h[BC]{0,0}\xc&=\langle[1]\rangle,\\
			\h[BC]{1,1}\xc&=\langle[\varphi^{1\bar{1}}],[\pp{2}{2}]\rangle,\\
			\h[BC]{2,1}\xc&=\langle[\pp{23}{2}+i\pp{13}{1}]\rangle,\\
			\h[BC]{1,2}\xc&=\langle[\pp{2}{23}-i\pp{1}{13}]\rangle,\\
			\h[BC]{2,2}\xc&=\langle[\pp{12}{12}]\rangle,\\
			\h[BC]{3,2}\xc&=\langle[\pp{123}{12}]\rangle,\\
			\h[BC]{2,3}\xc&=\langle[\pp{12}{124}]\rangle,\\
			\h[BC]{3,3}\xc&=\langle[\pp{123}{123}]\rangle,
		\end{align*}
		where all the representatives above are the $\Delta_{BC}$-harmonic ones. The other Bott-Chern cohomology groups are trivial.
		
		Note that 
		\[[\omega\wedge(\pp{23}{2}+i\pp{13}{1})]_A=\dfrac{-1+i}{2}[\pp{123}{12}]_A=\dfrac{-1+i}{2}[\dd\pp{13}{13}]_A=0.\]
		Hence $[\pp{23}{2}+i\pp{13}{1}]_{BC}\in\h[BC]{2,1}\xc$ is a primitive class but has no primitive representative. By conjugation, $[\pp{2}{23}-i\pp{1}{13}]_{BC}\in\h[BC]{1,2}\xc$ is again a primitive class but has no primitive representative.
		
	\section{Period map}
		In this section, we recall the definition of the period map and the local Torelli theorem. Tis section closely follows \cite{popovici2019holomorphic}.\\
		
		Fix a Hermitian metric $\omega$ on X. We have the Hodge star operator
		\[\star:\cinf_n(X,\cc)\longrightarrow\cinf_n(X,\cc)
		\]
		and $\star^2=(-1)^nid$, where $n=\dim_\cc X$. When $n$ is even , the eigenvalues of $\star$ is $1$ and $-1$.  When $n$ is odd, the eigenvalues of $\star$ is $i$ and $-i$. This induces a decomposition
		\[\cinf_n(X,\cc)=\Lambda^n_+\oplus\Lambda^n_-,
		\]
		where $\Lambda^n_+$ (resp. $\Lambda^n_-$) is the eigenspace corresponding to $1$ or $i$ (resp. $-1$ or $-i$).
		
		Since $\Delta=dd^*+d^*d$ commutes with $\star$, we have $\star(\mathcal{H}^n_\Delta\xc)=\mathcal{H}^n_\Delta\xc$. By $\mathcal{H}^n_\Delta\xc=\h{n}(X,\cc)$, we get a decomposition
		\[\h{n}(X,\cc)=\h[+]{n}(X,\cc)\oplus\h[-]{n}\xc.\]
		The Hodge-Riemann bilinear form can be defined on $\h{n}\xc$ without any assumption on $\omega$:
		\begin{align*}
			Q:\h{n}\xc\times\h{n}\xc&\longrightarrow\cc\\
		(\{\alpha\},\{\beta\})&\longmapsto(-1)^{\frac{n(n-1)}{2}}\int_X\alpha\wedge\beta.
		\end{align*}
		It is non-degenerate. Indeed for every class $\{\alpha\}$, if $\alpha$ is the $\Delta$-harmonic representative, then $\star\bar{\alpha}$ is also $\Delta$-harmonic. We have
		\[(-1)^{\frac{n(n-1)}{2}}Q(\{\alpha\},\{\star\bar{\alpha}\})=\int_X\alpha\wedge\star\bar{\alpha}=\lVert\alpha\rVert^2,\]
		which is not $0$ as long as $\alpha$ is not $ 0$.
		Therefore we define a non-degenerate sesquilinear form
		\begin{align*}
			H:\h{n}\xc\times\h{n}\xc&\longrightarrow\cc\\
			(\{\alpha\},\{\beta\})&\longmapsto(-1)^{\frac{n(n+1)}{2}}i^n\int_X\alpha\wedge\bar{\beta}.
		\end{align*}
		Then one can define the period domain as follows.
		\begin{definition}
			The period domain is defined as a subset of $\mathbb{P}\h[]{n}\xc$:
			\begin{itemize}
				\item If $n$ is even,
			\[D=\{\text{complex line } l\in\mathbb{P}\h[]{n}\xc\arrowvert\,\forall\varphi\in l\setminus\{0\},Q(\varphi,\varphi)=0\text{ and }H(\varphi ,\varphi)>0\};\]
			\item If $n$ is odd,
				\[D=\{\text{complex line } l\in\mathbb{P}\h[]{n}\xc\arrowvert\,\forall\varphi\in l\setminus\{0\},Q(\varphi,\varphi)=0\text{ and }H(\varphi ,\varphi)<0\}.\]
			\end{itemize}
		\end{definition}
		We prove that $\h[]{n,0}\xc$ is a subset of the period domain by the following two lemmas:
		\begin{lemma}
			Let $X$ be a compact complex $\ddb$-manifold, then
		\begin{itemize}
			\item if $n$ is even, $\h[]{n,0}\xc\subset\h[+]{n}\xc$;
			\item	if $n$ is odd, $\h[]{n,0}\xc\subset\h[-]{n}\xc$.
		\end{itemize}
		\end{lemma}
	\begin{proof}
		Because every element $\alpha\in\cinf_{n,0}\xc$ is primitive for bidegree reasons, we have $\star\alpha=i^{n(n+2)}\alpha$ by (\ref{voiprim}). Hence $\alpha$ is $\db$-closed if and only if $\star\alpha$ is $\db$-closed.
		We have $i^{n(n+2)}=1$ if $n$ is even, and $i^{n(n+2)}=-i$ if $n$ is odd. This proves the lemma.
	\end{proof}
		\begin{lemma}We have the following properties:\\
			$H(\{\alpha\},\{\alpha\})>0$ for every class $\{\alpha\}\in\h[+]{n}\xc\setminus\{0\}$,\\
			$H(\{\alpha\},\{\alpha\})<0$ for every class $\{\alpha\}\in\h[-]{n}\xc\setminus\{0\}$.
		\end{lemma}
	\begin{proof}
		If $n$ is even, for every class $\{\alpha\}\in\h[+]{n}\xc\setminus\{0\}$, we have $\star\alpha=\alpha$. Then $$H(\{\alpha\},\{\alpha\})=i^{n(n+2)}\int_X\alpha\wedge\bar{\alpha}=\int_X\alpha\wedge\star\bar{\alpha}=\lVert\alpha\rVert^2>0.$$
		
		If $n$ is odd, for every class $\{\alpha\}\in\h[+]{n}\xc\setminus\{0\}$,
		we have $\star\alpha=i\alpha$. We still have $$H(\{\alpha\},\{\alpha\})=i^{n(n+2)}\int_X\alpha\wedge\bar{\alpha}=\int_X\alpha\wedge\star\bar{\alpha}=\lVert\alpha\rVert^2>0.$$
		
		One can prove the second statement similarly.
	\end{proof}
		\begin{theorem}[Local Torelli Theorem {\cite[Thm 5.4]{popovici2019holomorphic}}]
			Let X be a compact Calabi-Yau $\ddb$-manifold of dimension n, and $\pi:\xx\longrightarrow B$ be its Kuranishi family. Then the associated period map
			\begin{align*}
				\mathcal{P}:B& \rightarrow D\subset \mathbb{P}H^n(X,\cc)\\
				t& \mapsto H^{n,0}(X_t,\cc)
			\end{align*}
			is a local holomorphic immersion.
		\end{theorem}
	\begin{proof}
		Because $X=X_0$ is Calabi-Yau $\ddb$-manifold, we know that $\xt$ is a Calabi-Yau $\ddb$-manifold for all $t$ in a neighbourhood of $0$. Then $H^{n,0}(X_t,\cc)$ is a point in $H^{n,0}(X_t,\cc)$.
		
		Besides, there is a holomorphic family of nowhere vanishing $n$-forms $(u_t)_t\in B$ on X, such that $u_t$ is a $J_t$-holomorphic $(n,0)$-form. Then we know $\h[]{n,0}(X_t,\cc)=\cc u_t$. Hence the period map $\mathcal{P}$ is holomorphic.
		
		If $\mathcal{P}$ is not a local immersion, then we can choose a point in $B$, say $0$, and a tangent vector $\dd/\dd t\in T_0B$, such that $(d\mathcal{P})_0(\dd/\dd t)=0$. Because $X$ is $\ddb$-manifold, we can choose a representative $\theta$ in $\rho(\dd/\dd t)\in\h[]{0,1}(X,T^{1,0}X)$ such that $d(\theta\lrcorner u_0)=0$, where $\rho$ is the Kodaira-Spencer map.
		
		By Ehresmann's lemma, we have a smooth family of diffeomorphisms $\Phi_t^{-1}:X_0\rightarrow X_t$. Taking a set of $J_t$-holomorphic coordinates $z_1(t),\cdots,z_n(t)$, we write $u_t=f_tdz_1(t)\wedge\cdots dz_n(t)$. Therefore, we have
		\begin{equation}\label{cv}
			\dfrac{\partial(\Phi_t^{-1})^*u_t}{\dd t}\rvert_{t=0}=\theta\lrcorner u_0+\dfrac{\partial f_t}{\dd t}\rvert_{t=0}dz_1(t)\wedge\cdots dz_n(t),
		\end{equation}
		where $v:=\dfrac{\partial f_t}{\dd t}\rvert_{t=0}dz_1(t)\wedge\cdots dz_n(t)$ is a $(n,0)$-form and $\theta\lrcorner u_0$ is a $(n-1,1)$-form. Then we know that $\theta=0$. Hence the period map $\mathcal{P}$ is a local immersion.

	\end{proof}
		\section{Metrics on $B$}\label{secmetrics}
		In this section we compare two versions of Weil-Petersson metric and the metric induced by the period map given in the previous section.\\
		
		We use the definition of $\omega$-minimal $d$-closed representative given in \cite{popovici2019holomorphic}. The $\omega$-minimal $d$-closed representative of a Dolbeaut cohomology class $[\beta]$ is $\beta_{min}=\beta+\db v_{min}$ , where $\beta$ is the $\Delta''$-harmonic representative in $[\beta]$ and $v_{min}$ is the solution of minimal $L^2$-norm of $\dd\beta=-\ddb v$.
		\begin{definition}
			Let $(u_t)_{t\in B}$ be a fixed holomorphic family of non-vanishing holomorphic $n$-forms on the fibres $(X_t)_{t\in B}$ and let $(\omega_t)_{t\in B_{[\omega]}}$ be a smooth family of SKT metrics on the fibres  $(X_t)_{t\in B_{_{[\omega]}}}$ such that $\omega_t\in\{\omega\}$ for any $t$ and $\omega_0=\omega$.  The Weil-Petersson metrics $G^{(1)}_{WP}$ and $G^{(2)}_{WP}$ are defined on $B_{[\omega]}$ by
			$$ G^{(1)}_{WP}([\theta_t],[\eta_t])\,:=\,\dfrac{\langle\langle \theta_t,\eta_t \rangle\rangle_{\omega_t}}{\int_{X_t}dV_{\omega_t}}$$
			$$ G^{(2)}_{WP}([\theta_t],[\eta_t])\,:=\,\dfrac{\langle\langle \theta_t\lrcorner u_t,\eta_t\lrcorner u_t \rangle\rangle_{\omega_t}}{i^{n^2}\int_{X_t} u_t\wedge \overline{u}_t}$$
			for any $t \in B_{_{[\omega]}}$, $ [\theta_t],[\eta_t]\in  H^{0,1}(X_t,T^{1,0}X_t)_{[\omega]}$. Here $\theta_t$ (resp. $\eta_t$) is chosen  such that $\theta_t\lrcorner u_t$ (resp. $\eta_t\lrcorner u_t$) is the $\omega_t$-minimal $d$-closed representative of the class $[\theta_t\lrcorner u_t]\in  H^{n-1,1}(X_t,\cc)$ (resp. $[\eta_t\lrcorner u_t]\in  H^{n-1,1}(X_t,\cc)$).
			
		\end{definition}
		\begin{remark}
			Denote the $(1,1)$-forms associated with $G_{WP}^{(1)}$, $G_{WP}^{(2)}$ by $\omega_{WP}^{(1)}$, $\omega_{WP}^{(2)}$. If $Ric(\omega_t)=0$ for all $t\in B_{[\omega]}$, we have $\omega_{WP}^{(1)}=\omega_{WP}^{(2)}$.
		\end{remark}
		
		Let $L=\mathcal{O}_{\mathbb{P}H^n(X,\cc)}(-1)$ be the tautological line bundle on $\mathbb{P}H^n(X,\cc)$.
		
		We set:
		\begin{align*}
			C_+&:=\{\{\alpha\}\in\h[]{n}\xc\arrowvert H(\{\alpha\},\{\alpha\})>0\};\\
			C_-&:=\{\{\alpha\}\in\h[]{n}\xc\arrowvert H(\{\alpha\},\{\alpha\})<0\};\\
			U^n_+&:=\{[l]\in\mathbb{P}H^n\xc\arrowvert l\text{ is a complex line such that }l\subset C_+\};\\
			U^n_-&:=\{[l]\in\mathbb{P}H^n\xc\arrowvert l\text{ is a complex line such that }l\subset C_-\}.
		\end{align*}
		Then Im$\mathcal{P}$ is a subset of $U^n_+$ when $n$ is even, Im$\mathcal{P}$ is a subset of $U^n_-$ when $n$ is odd.
		We can get a Hermitian fibre metric $h_L^+$ on $L\arrowvert_{U^n_+}$ from $H$. Then the associated Fubini-Study metric on $U^n_+$ is
		\[\omega^+_{FS}=-i\Theta_{h^+_L}(L\arrowvert_{U^n_+}).\]
		Similarly, we get the associated Fubini-Study metric on $U^n_-$:
		\[\omega^-_{FS}=-i\Theta_{h^-_L}(L\arrowvert_{U^n_-}).\]
		Then we have a Hermitian metric $\gamma$ on $B$:
		\begin{equation*}
			\gamma:=\begin{cases}
				\mathcal{P}^*\omega^+_{FS}&\text{if $n$ is even,}\\
				\mathcal{P}^*\omega^-_{FS}&\text{if $n$ is odd.}
			\end{cases}
		\end{equation*}
		\begin{lemma}\label{thirdmetric}
			
			The K\"ahler metric $\gamma$ defined on $B$ is independent of the choice of metrics on $(X_t)_{t\in B}$ and is explicitly given by the formula:
			\begin{align*}
				\gamma_t([\theta_t],[\theta_t])&=\dfrac{-\int_X (\theta_t\lrcorner u_t)\wedge\overline{(\theta_t\lrcorner u_t)}}{i^{n^2}\int_X u_t\wedge\overline{u_t}}=\dfrac{-H(\theta_t\lrcorner u_t,\theta_t\lrcorner u_t)}{i^{n^2}\int_X u_t\wedge\bar{u_t}}, \text{if $n$ is even,}\\
				\gamma_t([\theta_t],[\theta_t])&=\dfrac{-i\int_X (\theta_t\lrcorner u_t)\wedge\overline{(\theta_t\lrcorner u_t)}}{i^{n^2}\int_X u_t\wedge\bar{u_t}}=\dfrac{H(\theta_t\lrcorner u_t,\theta_t\lrcorner u_t)}{i^{n^2}\int_X u_t\wedge\bar{u_t}}, \text{if $n$ is odd,}
			\end{align*}
		for every $t\in B$ and every $[\theta_t]\in H^{0,1}(X_t,T^{1,0}X_t)$.
		\end{lemma}
		\begin{proof}
			Because $X$ is a Calabi-Yau manifold, there is a family of nowhere vanishing $J_t$-$(n,0)$-forms $(u_t)_{t\in B}$, such that $u_t$ is $J_t$-holomorphic.
			
			When $n$ is even, we have $|u_t|_{h^+_L}^2=H(u_t,u_t)$.
			Now we know that
			\[\omega_{FS}^+=-i\dd_t\db_t\log(H(u_t,u_t)).\] 
			Taking a class $[\theta]$ in $\h[]{0,1}(X,T^{1,0}X)$, assume $[\theta]$ is the image of $\frac{\dd}{\dd t}|_{t=0}$ under the Kodaira-Spencer map. Then		\[\gamma_0([\theta],[\theta])=-\dfrac{\dd^2\log(H(u_t,u_t))}{\dd t\dd \bar{t}}|_{t=0}=-\frac{\dd}{\dd t}(\dfrac{H(u_t,\dfrac{\dd u_t}{\dd t})}{H(u_t,u_t)})|_{t=0}.\]
			Because the left hand side of (\ref{cv}) is $d$-closed, $v$ is a $J_0$-holomorphic $(n,0)$-form. Thus, $v$ is $Cu_0$ for some constant $C$. So we have
			\begin{align*}
				&-\frac{\dd}{\dd t}(\dfrac{H(u_t,\dfrac{\dd u_t}{\dd t})}{H(u_t,u_t)})|_{t=0}\\=&-\dfrac{H(\dfrac{\dd u_t}{\dd t}|_{t=0},\dfrac{\dd u_t}{\dd t}|_{t=0})H(u_0,u_0)-H(\dfrac{\dd u_t}{\dd t}|_{t=0},u_0)H(u_0,\dfrac{\dd u_t}{\dd t}|_{t=0})}{H^2(u_0,u_0)}\\
				=&-\dfrac{H(\theta\lrcorner u_0,\theta\lrcorner u_0)H(u_0,u_0)+C\bar{C}H^2(u_0,u_0)-C\bar{C}H^2(u_0,u_0)}{H^2(u_0,u_0)}\\
				=&-\dfrac{H(\theta\lrcorner u_0,\theta\lrcorner u_0)}{H(u_0,u_0)}.
			\end{align*}
			The calculation of the case that $n$ is odd differs with only a $(-1)$ factor.
		\end{proof}

		Denote the space of global smooth forms of bidegree $(n-1,1)$ by $\Lambda^{n-1,1}$. Then we have two decompositions. The first one is Lefschetz decompositon:
		\[\Lambda^{n-1,1}=\Lambda^{n-1,1}_{prim}\oplus(\omega\wedge\Lambda^{n-2,0}).\]
		The second one is the decomposition into the eigenspaces of Hodge star operator:
		\[\Lambda^{n-1,1}=\Lambda^{n-1,1}_+\oplus\Lambda^{n-1,1}_-.\]
		\begin{lemma}\label{lemma17}
			These two decompositions coincide up to order. Specifically, we have
			
			\begin{itemize}
				\item $\Lambda^{n-1,1}_{prim}=\Lambda^{n-1,1}_-$ and $\omega\wedge\Lambda^{n-2,0}=\Lambda^{n-1,1}_+$ if $n$ is even,
			
			\item$\Lambda^{n-1,1}_{prim}=\Lambda^{n-1,1}_+$ and $\omega\wedge\Lambda^{n-2,0}=\Lambda^{n-1,1}_-$ if $n$ is odd.
			\end{itemize}
			
		\end{lemma}
		\begin{proof}
			For a primitive form $u$ of bidegree $(n-1,1)$, we have $\star u=(-1)^{n(n+1)/2}i^{n-2}u$. Therefore, we get $\Lambda^{n-1,1}_{prim}\subset\Lambda^{n-1,1}_-$  if $n$ is even, and $\Lambda^{n-1,1}_{prim}\subset\Lambda^{n-1,1}_+$ if $n$ is odd.
			
			Now, it suffices to prove that for a form $v$ of bidegree $(n-2,0)$, $\star(\omega\wedge v)=\omega\wedge v$ when $n$ is even, and $\star(\omega\wedge v)=-i\omega\wedge v$ when $n$ is odd.
			Firstly, for every form $u$ of bidegree $(n-1,1)$ we have a decomposition $u=u_{prim}+\omega\wedge u_1$. Besides, $v$ is primitive because it is of bidegree $(n-2,0)$. Hence, we get $\star v=\frac{i^{n(n-2)}}{2}v\wedge \omega^2$. Then
			\begin{equation*}
				\begin{aligned}
				\int_X u\wedge\star(\omega\wedge v)=&\langle\langle u,\omega\wedge\bar{v}\rangle\rangle\\
				=&\langle\langle \omega\wedge u_1,\omega\wedge\bar{v}\rangle\rangle\\
				=&2\langle\langle u_1,\bar{v}\rangle\rangle\\
				=&2\int_X u_1\wedge\star v\\
				=&i^{n(n-2)}\int_X u_1\wedge \omega^2\wedge v\\
				=&i^{n(n-2)}\int_X u\wedge \omega\wedge v.
				\end{aligned}
			\end{equation*}
			Therefore we have $\star(\omega\wedge v)=i^{n(n-2)}\omega\wedge v$.
		\end{proof}
		By Lemma \ref{lemma17}, for any $\theta\in\cinf_{0,1}(X,T^{1,0}X)$, we have the decomposition:
		\[\theta\lrcorner u=\theta'\lrcorner u+\omega\wedge\zeta.\]
		By orthogonality, we have
		\[G^{(2)}_{WP}([\theta_t],[\theta_t])\,=\,\dfrac{\langle\langle \theta_t\lrcorner u_t,\theta_t\lrcorner u_t \rangle\rangle}{i^{n^2}\int_{X_t} u_t\wedge \overline{u}_t}=\dfrac{\lVert\theta'_t\lrcorner u_t\rVert^2+2\lVert\zeta_t\rVert^2}{i^{n^2}\int_{X_t} u_t\wedge \overline{u}_t}.\]
		If $n$ is even, by Lemma \ref{lemma17}, we have $\star(\theta'\lrcorner u)=-\theta'\lrcorner u$ and $\star(\omega\wedge\zeta)=\omega\wedge\zeta$. As a consequence, we have
		\begin{equation*}
			\begin{aligned}
				\int_X (\theta_t\lrcorner u_t)\wedge\overline{(\theta_t\lrcorner u_t)}=&\int_X (\theta'_t\lrcorner u_t+\omega_t\wedge\zeta_t)\wedge\star(\overline{-\theta'_t\lrcorner u_t+\omega_t\wedge\zeta_t})\\
				=&-\lVert\theta'_t\lrcorner u_t\rVert^2+2\lVert\zeta_t\rVert^2.
			\end{aligned}
		\end{equation*}
		Then we have 
		\[\gamma_t([\theta_t],[\theta_t])=\dfrac{\lVert\theta'_t\lrcorner u_t\rVert^2-2\lVert\zeta_t\rVert^2}{i^{n^2}\int_{X_t} u_t\wedge \overline{u}_t}.\]
		Similarly, we get the same expression for $n$ odd.
		\begin{remark}
				For all $[\theta_t]\in H^{0,1}(X_t,T^{1,0}X_t)_{[\omega]}\setminus\{0\}$, we have
			\begin{equation*}
				(G_{WP}^{(2)}-\gamma)_t([\theta_t],[\theta_t])=\dfrac{4\|\zeta_t\|^2}{i^{n^2}\int_{X_t} u_t\wedge \overline{u}_t}\ge 0.
			\end{equation*}
		Hence if every class in $H^{n-1,1}_{prim}(X_t,\cc)$ has a $d$-closed and primitive representative, or equivalently, $H^{n-1,1}_{BC,prim}(X_t,\cc)$ has a primitive representative, we would get $G_{WP}^{(2)}=\gamma$. The polarised deformation of a K\"ahler manifold satisfies this condition in \cite{tian1987smoothness}.
		\end{remark}
		\bibliographystyle{alpha}
		\bibliography{bib}	
		\vspace*{2em}
		Institut de Math\'ematiques de Toulouse,\\
		Universit\'e Paul Sabatier,\\
		118 route de Narbonne, 31062 Toulouse, France\\
		Email: yi.ma@math.univ-toulouse.fr

\end{document}